\documentclass[12pt,a4paper]{amsart}
\usepackage{amscd,amsfonts,amsmath,amsthm,amssymb,verbatim,mathrsfs}
\usepackage[dvips]{graphicx}
\usepackage{graphics,hyperref}
\usepackage{enumitem}
\usepackage{psfrag}
\usepackage{pstricks}
\usepackage{latexsym}
\usepackage{multicol}

\psset{unit=0.7cm,linewidth=0.8pt,arrowsize=2.5pt 4}

\newpsstyle{fatline}{linewidth=1.5pt}
\newpsstyle{fyp}{fillstyle=solid,fillcolor=verylight}
\definecolor{verylight}{gray}{0.97}
\definecolor{light}{gray}{0.9}
\definecolor{medium}{gray}{0.85}

\def\NZQ{\mathbb}               
\def\NN{{\NZQ N}}

\def\ZZ{{\NZQ Z}}

\def\frk{\frak}               

\def\Phi{{\frk n}}
\def\Phi{{\frk N}}
\def\ab{\mathbf{a}}

\def\cb{\mathbf{c}}

\def\opn#1#2{\def#1{\operatorname{#2}}} 
\opn\chara{char} \opn\length{\ell} \opn\pd{pd} \opn\rk{rk}
\opn\projdim{proj\,dim} \opn\injdim{inj\,dim} \opn\rank{rank}
\opn\depth{depth} \opn\grade{grade} \opn\height{height}
\opn\embdim{emb\,dim} \opn\codim{codim} \opn\sgn{sgn}

\opn\Tr{Tr} \opn\bigrank{big\,rank}
\opn\superheight{superheight}\opn\lcm{lcm}
\opn\trdeg{tr\,deg}
\opn\reg{reg} \opn\lreg{lreg} \opn\ini{in} \opn\lpd{lpd}
\opn\size{size}\opn\bigsize{bigsize}
\opn\cosize{cosize}\opn\bigcosize{bigcosize}
\opn\sdepth{sdepth}\opn\sreg{sreg}
\opn\link{link}\opn\fdepth{fdepth}
\opn\Graver{Graver}
\opn\div{div} \opn\Div{Div} \opn\cl{cl} \opn\Cl{Cl} \opn\Cor{Cor}
\opn\Spec{Spec} \opn\Supp{Supp} \opn\supp{supp} \opn\Sing{Sing}
\opn\Ass{Ass} \opn\Min{Min}\opn\Mon{Mon} \opn\dstab{dstab} \opn\astab{astab}
\opn\Ann{Ann} \opn\Rad{Rad} \opn\Soc{Soc} \opn\Gr{Gr}
\opn\Im{Im} \opn\Ker{Ker} \opn\Coker{Coker} \opn\Am{Am}
\opn\Hom{Hom} \opn\Tor{Tor} \opn\Ext{Ext} \opn\End{End}
\opn\Aut{Aut} \opn\id{id} \opn\span{span}

\opn\nat{nat}
\opn\pff{pf}
\opn\Pf{Pf} \opn\GL{GL} \opn\SL{SL} \opn\mod{mod} \opn\ord{ord}
\opn\Gin{Gin} \opn\Hilb{Hilb}\opn\sort{sort} \opn\Gale{Gale}
\opn\aff{aff} \opn\conv{conv} \opn\relint{relint} \opn\st{st}   \opn\cone{cone}
\opn\lk{lk} \opn\cn{cn} \opn\core{core} \opn\vol{vol}
\opn\link{link} \opn\star{star}\opn\lex{lex} \opn\Gr{Gr}
\opn\gr{gr}

\def\pot#1#2{#1[\kern-0.28ex[#2]\kern-0.28ex]}

\opn\dirlim{\underrightarrow{\lim}}
\opn\inivlim{\underleftarrow{\lim}}
\let\union=\cup

\def\Implies{\ifmmode\Longrightarrow \else
        \unskip${}\Longrightarrow{}$\ignorespaces\fi}
\def\implies{\ifmmode\Rightarrow \else
        \unskip${}\Rightarrow{}$\ignorespaces\fi}
\def\iff{\ifmmode\Longleftrightarrow \else
        \unskip${}\Longleftrightarrow{}$\ignorespaces\fi}

\let\:=\colon
\newtheorem{Theorem}{Theorem}[section]
\newtheorem{Lemma}[Theorem]{Lemma}

\newtheorem{Proposition}[Theorem]{Proposition}
\newtheorem{Remark}[Theorem]{Remark}

\newtheorem{Example}[Theorem]{Example}

\newtheorem{Definition}[Theorem]{Definition}

\let\epsilon\varepsilon
\let\kappa=\varkappa
\makeatletter

\def\qed{\ifhmode\textqed\fi
      \ifmmode\ifinner\quad\qedsymbol\else\dispqed\fi\fi}
\def\textqed{\unskip\nobreak\penalty50
       \hskip2em\hbox{}\nobreak\hfil\qedsymbol
       \parfillskip=0pt \finalhyphendemerits=0}
\def\dispqed{\rlap{\qquad\qedsymbol}}

\opn\dis{dis}
\def\pnt{{\raise0.5mm\hbox{\large\bf.}}}

\opn\Lex{Lex}

\begin{document}


\title{The strongly robust simplicial complex  of monomial curves}
\author[1]{ Dimitra Kosta }
\author[2] {Apostolos Thoma}
\author[3] {Marius Vladoiu}
\thanks{Corresponding author: Dimitra Kosta}
\address{Dimitra Kosta, School of Mathematics, University of Edinburgh and Maxwell Institute for Mathematical Sciences, United Kingdom }
\email{D.Kosta@ed.ac.uk}

\address{Apostolos Thoma, Department of Mathematics, University of Ioannina, Ioannina 45110, Greece}
\email{athoma@uoi.gr}

\address{Marius Vladoiu, Faculty of Mathematics and Computer Science, University of Bucharest, Str. Academiei 14, Bucharest, RO-010014, Romania, and}
\address{``Simion Stoilow" Institute of Mathematics of the Romanian Academy, P.O. Box 1-764, 014700, Bucharest, Romania}
\email{vladoiu@fmi.unibuc.ro}


\begin{abstract} To every simple toric ideal $I_T$ one can associate the strongly robust
simplicial complex $\Delta _T$,  which determines the strongly robust property for
all ideals that have $I_T$ as their bouquet ideal.  We show that for the simple toric ideals of monomial curves in $\mathbb{A}^{s}$, the strongly robust simplicial complex $\Delta _T$ is either $\{\emptyset \}$ or
contains exactly one 0-dimensional face. In the case of monomial curves in $\mathbb{A}^{3}$,  the strongly robust simplicial complex $\Delta _T$ contains one 0-dimensional face if and only if the toric ideal $I_T$ is a complete intersection ideal with exactly two Betti degrees. Finally, we provide a construction to produce infinitely many strongly robust ideals with bouquet ideal the ideal of a monomial curve and show that they are all produced this way.
\end{abstract}
\maketitle

\section{Introduction}

  Let   $A\in\mathbb{Z}^{m\times n}$ be an integer matrix such that
  $\Ker_{\mathbb{Z}}(A)\cap \mathbb{N}^n=\{{\mathbf{0}}\}$.
  The toric ideal of $A$ is the ideal $I_A\subset K[x_1,\ldots,x_n]$
  generated by the binomials ${x}^{\mathbf{u}^+}-{x}^{{\mathbf{u}}^-}$
 where $K$ is a field, ${\mathbf{u}}\in\Ker_{\mathbf{Z}}(A)$
and ${\bf u}={\bf u}^+-{\bf u}^-$ is the unique expression of  ${\bf u}$ as a difference of two non-negative vectors with disjoint support, see \cite[Chapter 4]{St}.
A toric ideal is called robust if it is minimally generated by its Universal Gr{\"o}bner basis, where the  Universal Gr{\"o}bner basis is the union of all reduced   Gr{\"o}bner bases, see \cite{BR}.
  A {\it strongly robust} toric ideal  is a toric ideal $I_A$ for which  the Graver basis $\Gr(I_A)$
is a minimal system of generators, see \cite{S}. The condition
$\Ker_{\ZZ}(A)\cap \NN^n=\{{\mathbf{0}}\}$ implies that any minimal binomial generating set is contained in the Graver basis, see \cite[Theorem 2.3]{CTV2}. For strongly robust ideals then the Graver basis is the unique minimal system of generators and, thus, any reduced Gr{\"o}bner basis as well as the Universal Gr{\"o}bner basis are identical with the Graver basis, since all of them contain a minimal system of generators and they are  subsets of the Graver basis (see \cite[Chapter 4]{St}). We conclude that for a strongly robust toric ideal $I_A$ the
following sets are identical: the set of indispensable elements, any minimal system of binomial generators, any reduced Gr{\"o}bner basis, the Universal Gr{\"o}bner basis and the Graver basis. Therefore strongly robust toric ideals are robust.
The classical example of strongly robust ideals are the Lawrence ideals, see \cite[Chapter 7]{St}.
 There are several articles in the literature studying  robust related properties of ideals; see \cite{BZ, B, cng, CNG, SZ} for robust ideals, \cite{BR,  BBDLMNS} for robust toric ideals, \cite{G-MT, T} for generalized robust toric ideals and
\cite{    GP, KTV, PTV, PTV2, St,  S} for strongly robust toric ideals.

To characterize combinatorially the strongly robust property of toric ideals which have in common the same bouquet ideal $I_T$, in \cite{KTV}, we defined a simplicial complex, the strongly robust simplicial complex $\Delta_T$, the faces of which determine the strongly robust property. In particular, let $I_A$ be a toric ideal with bouquet ideal $I_T$, the ideal $I_A$ is strongly robust if and only
if the set $\omega$ of indices $i$, such that the $i$-th bouquet of $I_A$ is non-mixed,
is a face of  $\Delta_T$, see \cite[Theorem 3.6]{KTV}. Thus, understanding the strongly robust property of toric ideals $I_A$ is equivalent to understanding the strongly robust simplicial complex $\Delta_T$ for
simple toric ideals $I_T$. Simple toric ideals are ideals for which every bouquet is a singleton. Bouquet ideals are always simple. A method for the computation of the strongly robust simplicial complex $\Delta_T$ for a particular simple toric ideal $I_T$ was given in \cite[Theorem 3.7]{KTV}, however an interesting problem is to understand the strongly robust  simplicial complex $\Delta_T$ for classes of simple toric ideals. In this direction, we determine the strongly robust simplicial complex $\Delta_T$ for the simple toric ideals of monomial curves, which are toric ideals defined by $1\times s$-matrices, where $s\ge 3$.  For $s=2$, the toric ideal $I_T$ is principal and thus it is never simple but always strongly robust. For $s\ge 3$, the ideal of a monomial curve is never strongly robust, see Remark \ref{remark}.
 Toric ideals with bouquet ideal the toric ideal of a monomial curve
 include toric ideals of varieties with any dimension as well as varieties with any codimension greater than one.
 For example, a toric ideal with bouquet ideal the toric ideal $I_T$ of a monomial curve defined by the matrix $T=(24, 40, 41, 60, 80)$ is the ideal $I_A$ of a toric variety of dimension 7 and codimension 4
in Example \ref{E}, where we explain why such an ideal $I_A$ is strongly robust. In Section~\ref{genmat}, we provide infinitely many examples of strongly robust toric ideals  with bouquet ideal the toric ideal  of a monomial curve
and we prove that all such toric ideals are provided by this method.

In \cite{S}, Sullivant asked the question: does every strongly robust toric ideal $I_A$ of codimension $r$ have at least $r$ mixed bouquets? Since bouquets preserve the codimension and $s$ is the number of bouquets of $I_A$, Sullivant's question is equivalent to a question about the dimension of the strongly robust simplicial complex of its bouquet ideal $I_T$: is it true that simple toric ideals $I_T$ of codimension $r$ in the polynomial ring of $s$ variables have $\dim \Delta_T<s-r$? In Section~\ref{dimComplex}, we give an affirmative answer to the question of Sullivant for the simple toric ideals of monomial curves.

The structure of the paper is the following. In Section~\ref{section:prelim}, we present the notation, give definitions and previous results that will be required throughout the paper. Then, in Section~\ref{dimComplex}, we firstly proceed by showing that the dimension of the strongly robust  simplicial complex for a monomial curve $T$ is $\dim \Delta_{T} \leq 0$. Note that for monomial curves the codimension is $s-1$ thus $\dim \Delta_{T} \leq 0<1=s-(s-1)$, which agrees with Sullivant's conjecture. Section~\ref{section:CI_robustness} contains results that describe the circuits of $\Lambda(T)_{i}$
 which lead to Theorem~\ref{main1} that links the properties of complete intersection and strongly robustness for monomial curves in $\mathbb{A}^s$. Using this we give a full description of the strongly robust  simplicial complex in the case of monomial curves in $\mathbb{A}^3$ in Theorem~\ref{curveA3} and show that for monomial curves in $\mathbb{A}^s$ the strongly robust  simplicial complex is either $\{\emptyset \}$ or contains exactly one $0$-dimensional face in Theorem~\ref{Delta}.  In Section~\ref{primitiveRobustness}, we extend the notions of a primitive element as well as the Graver basis and give a necessary and sufficient condition in terms of primitive elements (or the Graver basis) on whether one specific element is the unique $0$-dimensional face of $\Delta_T$. Finally, in Section~\ref{genmat}, we use generalized Lawrence matrices to describe completely all matrices $A$ for which the toric ideal $I_A$ is strongly robust and which have bouquet ideal the toric ideal of a monomial curve.

 \section{Preliminaries} \label{section:prelim}

 Let $A=(\mathbf{a}_1, \ldots, \mathbf{a}_n)$ be an integer matrix in $\mathbb{Z}^{m \times n}$, with column vectors $\mathbf{a}_1,\ldots,\mathbf{a}_n$ and such that $\Ker_{\ZZ}(A)\cap \NN^n=\{{\mathbf{0}}\}$.
We say that ${\bf u}={\bf v}+_{c} {\bf w}$ is a conformal decomposition of the vector
${\bf u}\in  \Ker_\mathbb{Z}(A)$ if ${\bf u}={\bf v}+{\bf w}$ and ${\bf u}^+={\bf v}^++{\bf w}^+, {\bf u}^-={\bf v}^-+{\bf w}^-$,
where ${\bf v}, {\bf w} \in  \Ker_{\ZZ}(A)$. The conformal decomposition is called proper if both
 ${\bf v}$ and  ${\bf w}$ are not zero.
For the conformality, in terms of signs, the corresponding notation is the following:
    $+ = \oplus +_{c} \oplus$,
    $- = \ominus+_{c} \ominus$,
    $ 0 \ = \ 0 +_{c}  0 $.
 where the symbol $\ominus $ means that the corresponding integer is nonpositive and the symbol $\oplus $ nonnegative.
By $\Gr(A)$ we denote the set  of elements in $\Ker_\mathbb{Z}(A)$ that do not have a  proper conformal decomposition.
 A binomial ${\bf x}^{{\bf u}^+}-{\bf x}^{{\bf u}^-}\in I_A$
 is called \textit{primitive} if ${\bf u}\in \Gr(A).$ The set of the
primitive binomials is finite and it is called the \textit{Graver basis} of $I_A$ and is denoted by $\Gr(I_A)$, \cite[Chapter 4]{St}.

We recall from \cite[Definition 3.9]{HS} that for vectors ${\bf u},{\bf v},{\bf w}\in\Ker_{\mathbb{Z}}(A)$ such that ${\bf u}={\bf v}+{\bf w}$, the sum is said to be a \textit{semiconformal decomposition} of ${\bf u}$, written ${\bf u}={\bf v}+_{sc} {\bf w}$, if $v_i>0$ implies that $w_i\geq 0$, and $w_i<0$ implies that $v_i\leq 0$, for all $1\leq i\leq n$.  The decomposition is called {\it proper} if both ${\bf v}, {\bf w}$ are nonzero. The set of indispensable elements $S(A)$ of $A$ consists  of all nonzero vectors in $\Ker_{\ZZ}({A})$ with no proper semiconformal decomposition.
For the semiconformality, in terms of signs, the corresponding notation is the following:
 $+ = * +_{sc} \oplus $, $  - =  \ominus +_{sc}  * $, $  0 =  \ominus +_{sc}  \oplus  $, where the symbol $*$ means that it can take any value.

 A binomial ${\bf x}^{{\bf u}^+}-{\bf x}^{{\bf u}^-}\in I_A$
 is called \textit{indispensable} binomial if it belongs to the intersection of all minimal systems of binomial generators of $I_A$,  up to identification of opposite binomials. The set of indispensable binomials is  $S(I_A)=\{{\bf x}^{{\bf u}^+}-{\bf x}^{{\bf u}^-}|{\bf u}\in S(A)\}$ by \cite[Lemma 3.10]{HS} and \cite[Proposition 1.1]{CTV}.

  {\em Circuits} are irreducible binomials of a toric ideal $I_A$ with minimal support. In vector notation, a vector ${\bf u}\in \Ker_{\ZZ}(A)$ is called a circuit of the matrix $A$ if supp$({\bf u})$ is minimal and the components of ${\bf u}$ are relatively prime.

 To the matrix $A=(\mathbf{a}_1, \ldots, \mathbf{a}_n)$ we associate its Gale transform, which is the $n\times (n-r)$ matrix whose columns span the lattice $\Ker_{\ZZ}(A)$, where $r$ is the rank of $A$. We will denote the set of ordered row vectors of the Gale transform by $\{G({\bf a}_1), \dots, G({\bf a}_n)\}$. The vector ${\bf a}_i$ is called {\em free} if its Gale transform $G({\bf a}_i)$ is equal to the zero vector, which means that $i$ is not contained in the support of any element in $\Ker_{\mathbb{Z}}(A)$.
 The {\em bouquet graph} $G_A$ of $I_A$ is the graph on the set of vertices $\{{\bf a}_1,\dots, {\bf a}_n\}$, whose edge set $E_A$ consists of those $\{\mathbf{a}_i,\mathbf{a}_j\}$ for which $G({\bf a}_i)$ is a rational multiple of $G({\bf a}_j)$ and vice-versa. The connected components of the graph $G_A$ are called {\em bouquets}.

It follows from the definition that the free vectors of $A$ form one bouquet, which we call the {\em free bouquet} of $G_A$.  The non-free bouquets are of two types: {\em mixed} and {\em non-mixed}. A non-free bouquet is mixed if contains an edge $\{\mathbf{a}_i,\mathbf{a}_j\}$ such that $G({\bf a}_i)=\lambda G({\bf a}_j)$ for some $\lambda<0$, and is non-mixed if it is either an isolated vertex or for all of its edges $\{\mathbf{a}_i,\mathbf{a}_j\}$ we have $G({\bf a}_i)=\lambda G({\bf a}_j)$ with $\lambda>0$, see \cite[Lemma 1.2]{PTV}.

Let $ B_1,  B_2, \ldots,  B_s$ be the bouquets of $I_A$.
We reorder the vectors $\ab_1,\dots,\ab_n$ to $\ab_{11},\ab_{12}, \dots,\ab_{1k_1},\ \ab_{21},\ab_{22}, \dots,\ab_{2k_2},\ \dots, \ab_{s1},\ab_{s2}, \dots,\ab_{sk_s}$ in such a way that the first $k_1$ vectors belong to the bouquet $B_1$, the next $k_2$ to $B_2$ and so on up to the last $k_s$ that belong to the bouquet $B_s$. Note that $k_1+k_2+\dots +k_s=n$.
For each bouquet $B_i$ we define two vectors $\cb_{B_i}$ and $\ab_{B_i}$. If the bouquet $B_i$ is free then we set $\cb_{B_i}\in\ZZ^n$ to be any nonzero vector such that $\supp(\cb_{B_i})=\{i1,\dots , ik_i\}$ and with the property that the first nonzero coordinate, $c_{i1}$, is positive. For a non-free bouquet $B_i$ of $A$, consider  the Gale transforms of the elements in $B_i$. All the Gale transforms are nonzero, since the bouquet is non-free, and pairwise linearly dependent, since they belong to the same bouquet. Therefore, there exists a nonzero coordinate $l$ in all of them. Let $g_l=\gcd(G({\bf a}_{i1})_l, G({\bf a}_{i2})_l,\dots , G({\bf a}_{ik_i})_l)$, where $({\bf w})_l$ is the $l$-th component of a vector ${\bf w}$. Then ${\bf c}_{B_i}$  is the vector in $\ZZ^n$ whose $qj$-th component  is $0$ if $q\not= i$, and   $c_{ij}=\varepsilon_{i1}(G({\bf a}_{ij})_l)/g_l$, where $\varepsilon_{i1}$ represents the sign of the integer $G({\bf a}_{ij})_l$. Note that $c_{i1}$ is always positive. Then the vector $\ab_{B_i}$ (see \cite[Definition 1.7]{PTV}), is defined as $\ab_{B_i}=\sum_{j=1}^n ({c_{B_i})_j}\ab_j\in\ZZ^m$.

If $B_i$ is a  non-free bouquet of $A$, then $B_i$ is a mixed bouquet if and only if the vector ${\bf c}_{B_i}$ has a negative and a positive coordinate, and  $B_i$ is non-mixed if and only if the vector $\cb_{B_i}$ has all nonzero coordinates positive, see \cite[Lemma 1.6]{PTV}. The toric ideal $I_{A_B}$ associated to the matrix $A_B$, whose columns are the vectors $\ab_{B_i}$, $1\leq i\leq s$, is called the bouquet ideal of $I_A$.

Let ${\mathbf u}=(u_1, u_2, \ldots, u_s)\in  \Ker_{\ZZ}({A_B})$
then the linear map $$D({\mathbf u})=(c_{11}u_1, c_{12}u_1,\ldots ,c_{1k_1}u_1, c_{21}u_2, \ldots ,c_{2k_2}u_2, \ldots ,c_{s1}u_s, \ldots ,c_{sk_s}u_s),   $$
where all $c_{j1}, 1 \leq j \leq s$, are positive, is an isomorphism from $ \Ker_{\ZZ}({A_B})$ to $ \Ker_{\ZZ}(A)$, see \cite[Theorem 1.9]{PTV}.

The cardinality of the sets of different toric bases depends
only on the signatures of the bouquets and the bouquet ideal, see \cite[Theorem 2.3, Theorem 2.5]{KTV}
and \cite[Theorem 1.11]{PTV}.

Note that the bouquet ideal is simple: a toric ideal is called \textit{simple} if every bouquet is a singleton, in other words if $I_T\subset K[x_1,\dots , x_s]$ and has $s$ bouquets. The bouquet ideal of a simple toric ideal $I_A$ is $I_A$ itself. Let $I_A$ be the ideal of a monomial curve, where $A=(n_1, n_2,  \dots ,n_s)$ is an $1\times s$ matrix. In this case for any $i,j\in [s]$ there exists one circuit with support only $\{i, j\}$. Then for $s\geq 3$ and any two $n_k, n_l$ there exists a circuit that is zero on the $k^{th}$ component and nonzero on the
$l^{th}$ component and vice versa. Therefore the Gale transforms $G(n_k), G(n_l)$ are not the one multiple of the other. Thus all toric ideals of monomial curves  are simple if $s\geq 3$.

 \begin{Definition}
 {\em Let $I_T\subset K[x_1,\dots , x_s]$ be a simple toric ideal and $\omega \subset \{1, \dots , s\}$. A toric ideal $I_A$
 is called $T_{\omega}$-robust ideal if and only if
 \begin{itemize}
  \item the bouquet ideal of $I_A$ is $I_T$ and
  \item $\omega =\{i\in [s]| B_i \ \text{is non-mixed}\}.$
 \end{itemize}
 We denote by $$S_{\omega}(T)=\{{\bf u}\in \Gr(T)| \ D({\bf u})\in S(A)\}$$ and call $S_{\omega}(T)$ the $T_{\omega}$-indispensable set, where $I_A$ is an
 $T_{\omega}$-robust toric ideal and $S(A)$ is the set of indispensable elements of $A$.}
 \end{Definition}

The second part of the definition is correctly defined, since in \cite{KTV} we showed that the set of elements ${\bf u}$ which belong to $\Gr(T)$, such that $D({\bf u})$ is indispensable in a
 $T_{\omega}$-robust toric ideal $I_{A}$,
 does not depend on the $I_A$ chosen, but only on $T$ and $\omega$.

In \cite{KTV} we  introduced a simplicial complex, which determines the strongly robust property for toric ideals.

\begin{Definition}  {\em The set $\Delta _T =\{\omega\subseteq [s] \ | \ S_{\omega }(T)=\Gr(T)\}$ is called the {\it strongly robust complex} of $T$.}
\end{Definition}

 According to \cite[Corollary 3.5, Theorem 3.6]{KTV}, the  set $\Delta _T$ is a simplicial complex, which
 determines the strongly robust property for toric ideals.

\begin{Theorem} \label{face} \cite[Theorem 3.6]{KTV} Let $I_A$ be a $T_{\omega}$-robust toric ideal. The toric ideal $I_A$ is strongly robust if and only if $\omega$ is a face of the strongly
 robust complex $\Delta _T$.
\end{Theorem}

The following theorem provides a way to compute the strongly robust complex of a simple toric ideal $I_T$.
By $\Lambda (T)$ we denote the second Lawrence lifting of $T$, which is the $(m+s)\times 2s$ matrix
${\begin{pmatrix} T & 0 \\
I_s & I_s
\end{pmatrix}}.$ By $\Lambda (T)_{\omega}$ we denote the matrix taken from $\Lambda (T)$ by removing the $(m+i)$-th row and the $(s+i)$-th column for each $i\in \omega$.

\begin{Example} {\em  For $T=(n_1, n_2, n_3, n_4)$ and $\omega=\{ 3\}$, the $\Lambda (T)_{\omega}$
matrix is $${\begin{pmatrix} n_1 & n_2 & n_3 & n_4 & 0 & 0 & 0 \\
1 & 0 & 0 & 0 & 1 & 0 & 0 \\
0 & 1 & 0 & 0 & 0 & 1 & 0 \\
0 & 0 & 0 & 1 & 0 & 0 & 1
\end{pmatrix}}.$$}

\end{Example}
\begin{Theorem} \label{Lawrence}   \cite[Theorem 3.7]{KTV} The set $\omega$ is a face of the
 strongly robust complex $\Delta _T$ if and only if $I_{\Lambda (T)_{\omega}}$ is strongly robust.
\end{Theorem}

 \section{On the dimension of the strongly robust complex}
\label{dimComplex}

Suppose that $T= (n_1, n_2,n_3)$ with the property that  $\gcd{(n_1,n_2,n_3)}=1$.  For an element ${\bf u} = (u_1, u_2, u_3) \in \mathbb{N}^3 $, we define the $T$-degree of the monomial ${x}^{\bf u}$
to be $\deg_{T}({x}^{\bf u}) := u_1 n_1 + u_2 n_2 + u_3 n_3$. A vector $b$ is called a \textit{Betti $T$-degree} if $I_T$ has a minimal generating set of binomials containing an element of $T$-degree $b$. Betti $T$-degrees do not depend on the minimal set of binomial generators, \cite{CP, St}.

We know from J. Herzog \cite{Herzog}, that if $I_T$ is not complete intersection then the ideal $I_T$ is minimally generated by three binomials, while a complete intersection $I_T$ is minimally generated by two binomials.
 Let $c_i$ be the smallest multiple of $n_i$ that belongs to the semigroup generated by
$n_j,n_k$, where $\{i, j, k\}=\{1, 2, 3\}$.
Then there exist non-negative integers $c_{ij}, c_{ik}$ (not necessarily unique) such that
 $c_in_i = c_{ij}n_j + c_{ik}n_k$.
 We have the following cases.
\begin{itemize}
 \item[B3] If all six $c_{ij}\not =0$, then  $I_T$ is a  non complete intersection ideal,
  minimally generated by the three elements  $x_1^{c_1} - x_2^{c_{12}} x_3^{c_{13}}$, $x_2^{c_2}-x_1^{c_{21}}x_3^{c_{23}}$, $x_3^{c_3}-x_1^{c_{31}}x_2^{c_{32}}$.
 All minimal generators have full support, therefore the ideal is generic and thus all elements are indispensable, see \cite[Lemma 3.3, Remark 4.4]{PS2}. Being indispensable binomials implies that they have different $T$-degrees, see \cite{CKT}. Thus in the non complete intersection case we have 3 Betti $T$-degrees, $c_1n_1\not= c_2n_2\not =c_3n_3$.
 \item[B2-B1] If at least one $c_{ij}=0$, then  $I_T$ is a  complete intersection ideal
  and is generated by   $x_j^{c_j}-x_k^{c_k}$ and $x_i^{c_i} - x_j^{c_{ij}} x_k^{c_{ik}}$. There are two cases:
 \begin{itemize}
     \item[B2] the two binomials have different $T$-degrees, i.e. $c_jn_j= c_kn_k\not =c_in_i$. In this case we have two Betti $T$-degrees.
     \item[B1] the two binomials have the same $T$-degree, i.e.  $c_1n_1= c_2n_2 =c_3n_3$. In this case we have one Betti $T$-degree, see \cite{GOR}. It is easy to see that in the binomial $x_i^{c_i} - x_j^{c_{ij}} x_k^{c_{ik}}$ the monomial $x_j^{c_{ij}} x_k^{c_{ik}}$ can only be $x_j^{c_j}$ or $x_k^{c_k}$, otherwise one can find smaller multiples of $n_j$ or $n_k$ than $c_j$ or $c_k$ that belong to the semigroup generated by $n_i,n_k$ or $n_i,n_j$, contradicting the choice of $c_j$ or $c_k$. Thus both binomials are circuits. None of the circuits is indispensable since  any two of the three binomials   $x_1^{c_1}-x_2^{c_2}$, $x_1^{c_1} - x_3^{c_{3}}$  and  $x_2^{c_2} - x_3^{c_{3}}$ generate the ideal $I_T$.
 \end{itemize}
 \end{itemize}

 In both cases B1, B2, the circuit $x_j^{c_j}-x_k^{c_k}$ is the same as
 $x_j^{n_k^{\#}}-x_k^{n_j^{\#}}$, where $n_k^{\#}, n_j^{\#}$ are just the integers $n_k,n_j$
 divided by $g_{jk}= g.c.d(n_k,n_j)$. Thus $c_j=n_k^{\#}$ and $c_k=n_j^{\#}$.

\begin{Definition} \label{Def} {\em Let
     $T= (n_1, n_2,n_3)$, we say that $I_T$ is a  {\em complete intersection on $n_i$} if $c_jn_j= c_kn_k\not =c_in_i$. We say that  $I_T$ is a  {\em complete intersection on all} if $c_1n_1= c_2n_2 =c_3n_3$.}

\end{Definition}

 Remark that the condition $c_jn_j= c_kn_k\not =c_in_i$  implies that $I_T$ can be complete intersection on $n_i$ for at most one $n_i$.

\begin{Example} {\em
Let $T_1=(7, 15, 20)$, then $c_1=5, c_2=4, c_3=3$ and the Betti $T_1$-degrees are $5\cdot 7\neq 4\cdot 15=3\cdot 20$. Therefore $I_{T_1}$ is complete intersection on $n_1=7$. Let $T_2=(5, 6, 15)$, then $c_1=3, c_2=5, c_3=1$ and the Betti $T_2$-degrees are $3\cdot 5 = 1\cdot 15 \not =5\cdot 6$. Therefore $I_{T_2}$ is complete intersection on $n_2=6$. If $T_3=(6, 8, 11)$, then $c_1=4, c_2=3, c_3=2$, the Betti $T_3$-degrees are $4 \cdot 6 =3 \cdot 8 \not= 2 \cdot 11$, and therefore $I_{T_3}$ is complete intersection on $n_3=11$. Let $T_4=(6, 10, 15)$, then $c_1=5, c_2=3, c_3=2$ and there is only one Betti $T_4$-degree  $5 \cdot 6 = 3 \cdot 10 = 2 \cdot 15$. Therefore $I_{T_4}$ is complete intersection on all. Finally, for $T_5=( 3, 5, 7)$ we have $c_1=4, c_2=2, c_3=2$ and there are three Betti $T_5$-degrees: $4 \cdot 3\not = 2 \cdot 5 \not= 2 \cdot 7$. Therefore $I_{T_5}$ is not a complete intersection.
}

\end{Example}

 \begin{Proposition} \label{two indispensable}  Let
     $T= (n_1, n_2,n_3)$.
     In the toric ideal $I_T$ at least two of the three circuits are not indispensable.
 \end{Proposition}
 \begin{proof} In the case that $I_T$  is not complete intersection the toric ideal     is generated by three  binomials of full support \cite[Section 3]{Herzog}. Then the toric ideal is generic and by \cite[Lemma 3.3, Remark 4.4]{PS2} all three generators are indispensable. Since all generators have full support
none of them is a circuit.  Therefore none of the three circuits of the  toric ideal $I_T$ is indispensable. In the complete intersection case, the ideal $I_T$ has two minimal generators, one of which is always a circuit \cite[Proposition 3.5 and Theorem 3.8]{Herzog}.
If exactly one minimal generator was a circuit, then the other two circuits would not be indispensable. If both of the minimal generators were circuits then without loss of generality we can assume they would be of the form $x_i^a-x_j^b$, $x_i^c-x_k^d$. Namely, two of the monomials will be powers of the same variable. We can distinguish two cases.
\begin{itemize}
    \item Firstly, the two exponents $a$, $c$ could be different, so assume that $a < c$. In this case, according to \cite[Theorem 3.4]{CKT}, the binomial $x_i^c-x_k^d$ would not be indispensable, as $c$ would not be a minimal binomial $T$-degree. Therefore, in this case there exists at most one indispensable circuit, and since we have three circuits in total at least two would not be indispensable.
    \item In the second case, the two exponents $a, c$ are equal. Then both generators $x_i^a-x_j^b$ and $x_i^a-x_k^d$ would be of the same Betti $T$-degree. The ideal is generated by any two of the following three circuits $x_i^a-x_j^b$, $x_i^a-x_k^d$, $x_j^d-x_k^b$, since  $x_j^d-x_k^b=(x_i^a-x_j^b)-(x_i^a-x_k^d)$. Therefore, there is no indispensable binomial and in particular none of the circuits is indispensable. \qed
\end{itemize}
\end{proof}

\begin{Theorem} \label{dimension}
 Let $T=(n_1, n_2,  \dots ,n_s)$ then $\dim(\Delta _T)\leq 0.$
\end{Theorem}

\begin{proof}

Suppose on the contrary that $\dim(\Delta _T)> 0.$ Then there exist $i, j$ such that the edge $\{i,j\}$ belongs to the strongly robust complex $\Delta _T$. Therefore, the toric ideal $I_{\Lambda (T)_{\{i,j\}}}$
is strongly robust by Theorem~\ref{Lawrence}. Consider the ideal  $I_{(n_i, n_j, n_k)}$ for any $k \not = i, j$.
Let ${\bf c}$ be a non indispensable circuit of the toric ideal of $I_{(n_i, n_j, n_k)}$  different
from $(n_j^{'}, -n_i^{'},0)$, where $n_i^{'}, n_j^{'}$ are the $n_i, n_j$ divided by their greatest common divisor. Then without loss of generality ${\bf c}$ will be in the form ${\bf c}=(n_k^*,0, -n_i^*)$, where $n_i^*, n_k^*$ are the $n_i, n_k$ divided by their greatest common divisor.
We know that there always exists such a circuit ${\bf c}$, since by the Proposition \ref{two indispensable} we have that at least two of the three circuits of $I_{(n_i, n_j, n_k)}$ are not indispensable. As the circuit ${\bf c}$
is not indispensable, it has a proper semiconformal decomposition into two vectors with the following pattern of signs $(n_k^*,0, -n_i^*)=( \ast , - , \ominus )+_{sc}( \oplus , + , \ast )$.   The first $\ast$ is a positive number and the second $\ast$ is a negative number, since $\Ker_{\ZZ}(n_i,n_j,n_k)\cap \NN^3=\{{\mathbf{0}}\}$. Namely, $(n_k^*,0, -n_i^*)=( a , -b , -c )+_{sc}( d , b , -e )$, where $a,b,c,d,e\in \NN$ and $abe \not =0$, so this is a proper semiconformal decomposition. We will show below that this lifts into a proper semiconformal decomposition in $I_{\Lambda (T)_{\{i,j\}}}$.  Indeed, in the toric ideal  $I_{\Lambda (T)_{\{i,j\}}}$ we have
$$( 0,\ldots,n_k^*,\ldots,0,\ldots,-n_i^*,\ldots,n_i^*,\ldots,0)=$$ $$(0,\ldots,a,\ldots,-b,\ldots,-c,\ldots,c,\ldots,0)+_{sc}(0,\ldots,d,\ldots,b,\ldots,-e,\ldots,e,\ldots,0), $$
where the only nonzero components are in the $i^{th}$,$k^{th}$ and $(s+k-2)^{th}$ positions in the first vector and in the $i^{th}$, $j^{th}$, $k^{th}$ and $(s+k-2)^{th}$ positions in the last two. This decomposition is proper semiconformal, since $(n_k^*,0, -n_i^*)=( a , -b , -c )+_{sc}( d , b , -e )$ is proper.
We observe that the element $( 0,\ldots,n_k^*,\ldots,0,\ldots,-n_i^*,\ldots,n_i^*,\ldots,0)=D((0,\ldots,n_k^*,\ldots,0,\ldots,-n_i^*,\ldots,0))$ is a circuit of $I_{\Lambda (T)_{\{i,j\}}}$, since $D$ maps circuits to circuits \cite[Theorem 1.11]{PTV}, and is not an indispensable element in $I_{\Lambda (T)_{\{i,j\}}}$, since it admits a proper semiconformal decomposition.
However, we know that circuits are always contained in the Graver basis \cite[Proposition 4.11]{St}, so this means that the Graver basis $\Gr\left(\Lambda (T)_{\{i,j\}}\right)$ is not equal to the set of indispensable elements $S\left(\Lambda (T)_{\{i,j\}}\right)$. Therefore, the toric ideal $I_{\Lambda (T)_{\{i,j\}}}$
is not strongly robust, a contradiction. We conclude that $\dim(\Delta _T)\leq 0.$ \qed

\end{proof}

  In \cite[Corollary 1.3]{S}, Sullivant proved that strongly robust codimension 2 toric ideals
 have at least 2 mixed bouquets. For the strongly robust complex, this result means that $\dim(\Delta _T) < s-2.$ In the same paper, Sullivant poses a stronger question which can be translated to the following question: for every simple codimension $r$ toric ideal $I_T$, is it true that $\dim(\Delta _T)< s-r$? Theorem \ref{dimension} proves that the answer to this question is affirmative for
 the simple toric ideals of monomial curves. Note that toric ideals of monomial curves have codimension $r=s-1$.

\section{Complete Intersection and strongly robustness}
\label{section:CI_robustness}

\begin{Proposition} \label{conformal} Let $T=(n_1, n_2,  \dots ,n_s)$ and  ${\bf u},{\bf v},{\bf w}\in\Ker_{\mathbb{Z}}(T)$. If  $D({\bf u})=D({\bf v})+_{sc}D({\bf w})$ in $\Ker_{\mathbb{Z}}\left(\Lambda (T)_{\{i\}}\right)$, then $[{\bf u}]^i=[{\bf v}]^i+_{c}[{\bf w}]^i$, where $[{\bf u}]^i$ is the vector obtained from ${\bf u}$ by deleting its $i^{th}$ component.
\end{Proposition}
\begin{proof}
    Let $j \in \{1, \ldots,s\}$ be such that $j\not = i$. Then, for the vector $D({\bf u})$ in the kernel $\Ker_{\mathbb{Z}}\left(\Lambda (T)_{\{i\}}\right)$, one of the components is equal to $u_j$ and another is $-u_j$. Similarly, the corresponding two components of $D({\bf v}), D({\bf w}) \in \Ker_{\mathbb{Z}}\left(\Lambda (T)_{\{i\}}\right)$ are
    $v_j, -v_j$ and $w_j, -w_j$ respectively. The semiconformal decomposition $D({\bf u})=D({\bf v})+_{sc}D({\bf w})$, implies that  on those components we have
   \begin{eqnarray}\label{sc_1}
 (u_j) & = & (v_j)+_{sc} (w_j), \hspace{0.5cm} \\
\label{sc_2}
   (-u_j) & = & (-v_j)+_{sc} (-w_j). \hspace{0.5cm}
  \end{eqnarray}

    If $u_j\geq 0$, then $w_j\geq 0$ by (\ref{sc_1}), while $-v_j\leq 0$ by (\ref{sc_2}). Therefore, both $v_j, w_j$ are non-negative and so the sum $(u_j)=(v_j)+_{c} (w_j)$ is conformal. If on the other hand $u_j\leq 0$, then $v_j\leq 0$ by (\ref{sc_1}) and $-w_j\geq 0$ by (\ref{sc_2}). Therefore, both $v_j, w_j$ are non-positive and the sum $(u_j)=(v_j)+_{c} (w_j)$
    is again conformal.  \qed

\end{proof}
\begin{Lemma} \label{Lemma}
Any circuit of $\Lambda (T)_{\{i\}}$ that has $i$ in its support is  indispensable.
\end{Lemma}
\begin{proof} Due to the one-to-one correspondence between circuits of $\Lambda (T)_{\{i\}}$ and circuits of $T$, a circuit of $\Lambda (T)_{\{i\}}$ can be written as $D({\bf u})$, where ${\bf u}$ is a circuit of $T$, see \cite[Theorem 1.11]{PTV}. Therefore, $D({\bf u})$ has the form
 $D(0, \ldots, 0,n_i^{\#}, 0, \ldots, 0,-n_j^{\#}, 0, \ldots, 0)$, where $n_i^{\#}, n_j^{\#}$ are the $n_i, n_j$ divided by their greatest common divisor,  $n_i^{\#}$ is the $j^{th}$ component of ${\bf u}$ and $-n_j^{\#}$ is the $i^{th}$ component of ${\bf u}$.
 Let $D({\bf u})=D({\bf v})+_{sc}D({\bf w})$ be a semi-conformal decomposition of $D({\bf u})$. Then, by Proposition \ref{conformal}, we have that $[{\bf u}]^i=[{\bf v}]^i+_{c}[{\bf w}]^i$.
As the only conformal decomposition of $0$ is $0+0$, the only possible non zero component
of each of the three vectors $[{\bf u}]^i, [{\bf v}]^i, [{\bf w}]^i$ is the $j^{th}$ and we have that
$n_i^{\#}= v_j+_cw_j$. Thus, both $v_j, w_j$ are non negative and at least one is positive.
Without loss of generality we can assume that $v_j >0$, so ${\bf v}$ is not zero and  $[{\bf v}]^i$
has only one element $v_j$ in its support. This means that ${\bf v}$ has minimal support $\{i, j\}$
and so it is a  multiple of the circuit ${\bf u}$. Therefore, ${\bf v}=l{\bf u}$ and $l \geq 1$,
since $v_j$ is positive. Thus, $n_i^{\#}= v_j+w_j \geq v_j = l n_i^{\#}$, therefore $l=1$, ${\bf v}={\bf u}$ and for ${\bf w}$ the only option is that ${\bf w}={\bf 0}$, thus $D({\bf w})={\bf 0}$, leading  to a non-proper decomposition. This proves that $D({\bf u})$
is indispensable. \qed

\end{proof}

\begin{Lemma} \label{circuit-ind} Let $T=(n_1, n_2,  \dots ,n_s)$ and ${\bf u}$ be the circuit with support on $j,k\in [s]$. Then the circuit $D({\bf u})$ is indispensable in   $\Lambda (T)_{\{i\}}$ if and only if $I_{(n_i, n_j, n_k)} $ is a complete intersection on $n_i$.
\end{Lemma}

\begin{proof} The circuit with support on $j,k$ is
 ${\bf u}=(0, \ldots, 0,n_k^{\#}, 0, \ldots, 0,-n_j^{\#}, 0, \ldots, 0)$, where the two nonzero elements $n_k^{\#}, n_j^{\#}$ are in the $j^{th}$ and $k^{th}$ position respectively and $n_k^{\#}, n_j^{\#}$ are the $n_k, n_j$ divided by their greatest common divisor. To prove one implication, let $I_{(n_i, n_j, n_k)} $ be a complete intersection on $n_i$. Then $n_k^{\#}=c_j$ and $n_j^{\#}=c_k$ and thus g.c.d$(c_j, c_k)=1$. Suppose that the circuit $D({\bf u})$ is not indispensable in   $\Lambda (T)_{\{i\}}$ and let  $D({\bf u})=D({\bf v})+_{sc}D({\bf w})$ be a proper semiconformal decomposition of $D({\bf u})$. Then, by Proposition \ref{conformal}, we have that $[{\bf u}]^i=[{\bf v}]^i+_{c}[{\bf w}]^i$. Also, coordinate wise $c_j=v_j+_cw_j=a+b$ and
 $c_k=v_k+_c w_k=-c-d$, where $a,b,c,d\in \mathbb{N}$.
Moreover, from the semiconformal decomposition of $D({\bf u})$, we have that $0 = v_i+_{sc} w_i$, so $v_i=-e, w_i=e$, where $e\in \mathbb{N}$. The rest of the components of ${\bf u}, {\bf v}, {\bf w}$ are zero,  since the only conformal decomposition of $0$ is $0+0$, by Proposition \ref{conformal}.
 Then, since
 ${\bf u}, {\bf v}, {\bf w}\in \Ker_{\mathbb{Z}}(T)$ and
 \begin{eqnarray*}
{\bf v} & = &  (0, \ldots, 0, ~a , 0, \ldots, -e, \ldots, 0, -c , 0, \ldots, 0), \\
{\bf w} & =  & (0, \ldots, 0, ~b , 0, \ldots, \hspace{0.3cm}e, \ldots, 0, -d , 0, \ldots, 0) ,
 \end{eqnarray*}
  we have that $an_j-en_i-cn_k=0$ and
  $bn_j+en_i-dn_k=0$, where the possible nonzero components of ${\bf v}, {\bf w}$ are in the $j^{th}$, $i^{th}$ and $k^{th}$ positions. This implies that $an_j=cn_k+en_i$ and we distinguish two cases: $a=0$ and $a>0$.
  If $a=0$, then $an_j=0=cn_k+en_i$ implies that $c=0$ and $e=0$, which means that ${\bf v}={\bf 0}$ and $D({\bf v})={\bf 0}$, a contradiction.
  In the case that $a>0$, then $an_j$ belongs to the semigroup    $<n_k,n_i>$. However, $c_jn_j$ is the smallest multiple of $n_j$ that belongs to the semigroup $<n_k,n_i>$. Thus, $a\ge c_j=a+b$, which implies that $a=c_j$
   and $b=0$. We similarly argue for $dn_k=bn_j+en_i$. If $d=0$, then  $dn_k=0=bn_j+en_i$ implies $b=0$ and $e=0$. This means that ${\bf w}={\bf 0}$ and $D({\bf w})={\bf 0}$, a contradiction.
   In the case that $d>0$, then $dn_k$ belongs to the semigroup  $<n_i,n_j>$. However, $c_kn_k$ is the smallest multiple of $n_k$ that belongs to the semigroup $<n_i,n_j>$.  Thus, $d\ge c_k=c+d$, which implies that $d=c_k$
   and $c=0$.

   In conclusion, from $an_j=cn_k+en_i$ and $dn_k=bn_j+en_i$, we have that $b=c=0$, $a=c_j$, $d=c_k$ and thus $c_jn_j=en_i$ and  $c_kn_k=en_i$. The equation $c_jn_j=en_i$ implies that g.c.d$(c_j, e)=1$, since otherwise a smaller
   multiple of $n_j$ would belong to the semigroup $<n_k,n_i>$. Then $e$ divides $n_j$ and thus $n_i$ is a multiple of $c_j$. Similarly from $c_kn_k=en_i$ we have
   g.c.d$(c_k, e)=1$. Then $e$ divides $n_k$ and thus $n_i$ is a multiple of $c_k$.
   Since g.c.d$(c_j, c_k)=1$ we have that $n_i=\lambda c_j c_k$. Then  $c_jn_j=en_i$ implies
   $n_j=\lambda e c_k$ and $c_kn_k=en_i$ implies $n_k=\lambda e c_j$.

   Summarizing we have $n_i=\lambda c_j c_k$, $n_j=\lambda e c_k$,  $n_k=\lambda e c_j$,  g.c.d$(c_j, e)=1$, g.c.d$(c_k, e)=1$ and g.c.d$(c_j, c_k)=1$. Recall that $c_in_i$ is the smallest multiple of $n_i$ that belongs to the semigroup generated by $n_j, n_k$. Then $c_in_i=c_{ij}n_j+c_{ik}n_k$ implies that $c_i\lambda c_j c_k=c_{ij}\lambda e c_k+c_{ik}\lambda e c_j$. We conclude that $e$ divides $c_i$, since  g.c.d$(c_j, e)=1$, g.c.d$(c_k, e)=1$. Therefore $e\le c_i$, but from the defining property of $c_i$ and the fact that $en_i=c_jn_j$ we have $e\ge c_i$. Thus $c_i=e$ and $c_in_i=c_jn_j=c_kn_k$.
   This means that $I_{(n_i, n_j, n_k)} $ is complete intersection on all, a contradiction. Thus, $D({\bf u})$ is indispensable  in   $\Lambda (T)_{\{i\}}$.

   To prove the other implication, suppose now that $I_{(n_i, n_j, n_k)} $ is not a complete intersection on $n_i$. It follows then that either  is a complete intersection on all or is not complete intersection at all.
If it is complete intersection on all then $c_jn_j=c_in_i=c_kn_k$. Then, we have the proper semi-conformal decomposition
${\bf u}= {\bf v}+_{sc}{\bf w}$, where
 \begin{eqnarray*}
{\bf u} & = &  (0, \ldots, ~c_j~,\ldots , 0, ~0~, 0, \ldots, 0,-c_k, 0, \ldots, 0),\\
{\bf v} & = &  (0, \ldots, ~c_j~,\ldots , 0, -c_i, 0, \ldots, 0, \hspace{0.3cm} 0~,~ 0, \ldots, 0), \\
{\bf w} & =  & (0, \ldots, ~0~,\ldots , 0, ~c_i~, 0, \ldots, 0, -c_k, 0, \ldots, 0) \text{ , }
 \end{eqnarray*}
  where the components that can be nonzero in at least one vector are in the $j^{th}$, $i^{th}$ and $k^{th}$-positions.
   This implies that $D({\bf u})=D({\bf v})+_{sc}D({\bf w})$, since on all components except the one in the $i^{th}$ position the sum is conformal. Thus $D({\bf u})$ is not indispensable  in   $\Lambda (T)_{\{i\}}$, a contradiction.

  In the case that $I_{(n_j, n_i, n_k)} $ is not complete intersection, then the circuit $(c_j, 0,-c_k)$ is not indispensable in $I_{(n_j,n_i,n_k)}$, therefore it has a proper semiconformal decomposition. Note that in terms of signs
   $(c_j, 0,-c_k)=(  \ast, \ominus, \ominus )+_{sc}( \oplus ,  \oplus , \ast )={\bf v}+_{sc}{\bf w}$.
   The first $\ast$ is $+$ and the second $\ast$ is $-$, since $\Ker_{\ZZ}(n_j,n_i,n_k)\cap \NN^3=\{{\mathbf{0}}\}$.  Thus, the sum is conformal in the first and the last component. Then
   \[
    D((0, \ldots, c_j,\ldots , 0, 0, \ldots,-c_k, 0, \ldots, 0))=
    \]
   \[ D((0, \ldots, v_j,\ldots , v_i, 0, \ldots, v_k, 0, \ldots, 0))+_{sc}D((0, \ldots, w_j,\ldots , w_i, 0, \ldots, w_k, 0, \ldots, 0)) .
   \]

   Thus, $D({\bf u})$ is not indispensable  in   $\Lambda (T)_{\{i\}}$, a contradiction, and therefore the other implication is also proved. \qed

\end{proof}
 The following theorem shows that if $\{ i\}$ is a face of the strongly robust complex $\Delta _T$, then $n_i$ has a very special property. As we will see in Theorem \ref{Delta}, if $n_i$ satisfies this property, then $n_i$ is unique.

\begin{Theorem} \label{main1} Let $T=(n_1, n_2,  \dots ,n_s)$.
If the strongly robust complex $\Delta _T$ contains $\{ i\}$ as a face then
for every $j,k\in [s]$ with $j,k \neq i$, $I_{(n_i, n_j, n_k)} $ is a complete intersection on $n_i$.
\end{Theorem}

\begin{proof}
Suppose that $\{ i\}$ is a face of $\Delta _T$, then $\Lambda (T)_{\{i\}}$ is strongly robust. Then all circuits are indispensable, since for strongly robust toric ideals the set of indispensable elements is equal to the Graver basis and the latter contains all circuits. Therefore, by Lemma \ref{circuit-ind},
$I_{(n_i, n_j, n_k)} $ is a complete intersection on $n_i$. \qed

\end{proof}

\begin{Theorem} \label{Delta}
Let $T=(n_1, n_2,  \dots ,n_s)$.
The strongly robust complex $\Delta _T$ is either $\{\emptyset \}$ or $\{\emptyset, \{i\}\}$
for exactly one $i\in [s]$.

\end{Theorem}
  \begin{proof} By Theorem \ref{dimension} $\dim(\Delta _T)\leq 0,$ thus $\Delta _T=\{\emptyset \}$ or $\Delta _T=\{\emptyset \}\union \{\{i\}| i\in \Sigma\subset [s] \}.$ We claim that $\Sigma$ contains at most one element. Suppose not, and let $i,j\in \Sigma$, $i\not =j$. Take any $k\in [s]$ different from $i, j$ and apply two times Theorem \ref{main1}. We have that $I_{(n_i, n_j, n_k)} $ is a complete intersection on $n_i$
  and $n_j$, and thus a contradiction, see remark after Definition \ref{Def}. Thus the set $\Sigma$ can have at most one element.  \qed

  \end{proof}

\begin{Remark} \label{remark} {\em It follows from Theorem \ref{Delta} that toric ideals of monomial curves in  $A^s$ for $s\ge 3$ are never strongly robust,
since in this case $I_T$ is a simple toric ideal, therefore all of its bouquets are not mixed. Thus, $I_T$ is a $T_{[s]}$-robust ideal and $[s]\not \in \Delta_T$. The fact that the toric ideal of a monomial curve is not robust for $s\geq 3$, thus also not strongly robust, was also noticed in \cite[Corollary 4.17]{G-MT}.}
\end{Remark}

Although the converse of Theorem \ref{main1} is not true in general, it is true for $1\times 3$ matrices, as the following Theorem shows.

\begin{Theorem}
\label{curveA3}
  Let $T=(n_1, n_2, n_3)$. We have the following three cases
  \begin{itemize}
   \item if $I_T$ is not complete intersection, then $\Delta _T$ is the empty complex;
   \item if $I_T$ is complete intersection on $n_i$, for an $i\in [3]$,  then  $\Delta _T=\{\emptyset, \{i\}\}$;
   \item if $I_T$ is complete intersection on all, then $\Delta _T$ is the empty complex.

  \end{itemize}

 \end{Theorem}

 \begin{proof}
 By Proposition \ref{Delta}, we have that the strongly robust complex $\Delta _T$ is either $\{\emptyset \}$ or $\{\emptyset, \{i\}\}$
for one $i\in [3]$. By Theorem \ref{main1}, if $ \{i\}\in \Delta _T $, then $I_{(n_i, n_j, n_k)} $ is a complete intersection on $n_i$. Therefore it remains to show that if $I_T$ is complete intersection on $n_i$, for an $i\in [3]$,  then  $\Delta _T=\{\emptyset, \{i\}\}$, or equivalently that $\Lambda (T)_{\{i\}}$ is strongly robust. Without loss of generality we may suppose that $i=1$. By Lemma \ref{Lemma} the two circuits with 1 in their support are indispensable, while by Lemma \ref{circuit-ind} the remaining circuit with support on $\{2,3\}$ is indispensable in
 $\Lambda (T)_{\{1\}}$, as $I_T$ is a complete intersection on $n_1$.  Thus it remains to prove that elements  ${\bf u}$ in the $\Gr(T)$ with full support are indispensable in  $\Lambda (T)_{\{1\}}$. Taking ${\bf u}$
 or $-{\bf u}$, we can suppose that the first component of ${\bf u}$ is positive.

 There are three cases then for ${\bf u}$: $(1)\ {\bf u}=(a, -b, -c)$,  $(2)\ {\bf u}=(a, b, -c)$, and  $(3)\ {\bf u}=(a, -b, c)$, where $a, b, c\in \mathbb{N}$.

$(1)\ {\bf u}=(a, -b, -c)$.  Let $D({\bf u})=D({\bf v})+_{sc}D({\bf w})$ be a semiconformal decomposition
of $D({\bf u})$ in $\Lambda (T)_{\{1\}}$. Then, from Proposition \ref{conformal},  we have $[{\bf u}]^1=[{\bf v}]^1+_{c}[{\bf w}]^1$, but $[{\bf u}]^1=(-b, -c)$ and the sum being conformal implies that all components of
 $[{\bf v}]^1, [{\bf w}]^1$ are non positive. Since ${\bf v}, {\bf w}\in \Ker_{\mathbb{Z}}(T)$, this means
 that their first component is non negative. However, then the sum $D({\bf u})=D({\bf v})+_{c}D({\bf w})$
 is conformal, and $D({\bf u})\in \Gr(\Lambda (T)_{\{1\}})$ since ${\bf u}\in \Gr(T)$, \cite{PTV}. As elements of the Graver basis as characterised as those with no proper conformal decomposition,
one of the $D({\bf v}), D({\bf w})$  is zero and thus $D({\bf u})$ is indispensable.

$(2)\ {\bf u}=(a, b, -c)$. Let $D({\bf u})=D({\bf v})+_{sc}D({\bf w})$ be a semiconformal decomposition
of $D({\bf u})$ in $\Lambda (T)_{\{1\}}$. Then, from Proposition \ref{conformal},  we have $[{\bf u}]^1=[{\bf v}]^1+_{c}[{\bf w}]^1$. However, $[{\bf u}]^1=(b, -c)$ and the sum being conformal implies that the second component of each of the ${\bf v}, {\bf w}$ is non negative, while the third component is non positive. Taking into account that the sum $D({\bf u})=D({\bf v})+_{sc}D({\bf w})$ is semiconformal, we get that the sign patent of ${\bf v}, {\bf w}$
is ${\bf v}=( \ast, \oplus, \ominus )$ and ${\bf w}=(\oplus ,\oplus , \ominus)$. If $\ast=\oplus$,
then ${\bf u}$ will have a conformal decomposition and  since ${\bf u}\in \Gr(T)$, one of the ${\bf v}, {\bf w}$  is zero. Which implies that
one of the $D({\bf v}), D({\bf w})$  is zero and thus $D({\bf u})$ is indispensable.

 If $\ast=\ominus$ then $ {\bf u}=(a, b, -c)={\bf v}+{\bf w}=(-a_1, b_1, -c_1)+(a_2, b_2, -c_2), $ where
$a_1, b_1, c_1, a_2, b_2, c_2\in \mathbb{N}$. Then, $-a_1n_1+b_1n_2-c_1n_3=0$ and $a_2n_1+b_2n_2-c_2n_3=0$, since ${\bf v}, {\bf w}\in \Ker_{\mathbb{Z}}(T)$.
In the first equation, $b_1=0$ implies that $a_1=0=c_1$ thus ${\bf v}={\bf 0}$ and the proof is complete, as $D({\bf u})$ is indispensable.
Otherwise,  $b_1n_2$ belongs to the semigroup generated by $n_1,n_3$ and  $(n_1, n_2, n_3) $ is a complete intersection on $n_1$ implies $b=b_1+b_2\ge b_1\ge n_3^{\#}.$ Similarly from the second equation we get that either ${\bf w}={\bf 0}$ or $c=c_1+c_2\ge c_2\ge n_2^{\#}.$ But then the proper sum $(a,b,-c)=(a,b-n_3^{\#},-c+n_2^{\#})+(0, n_3^{\#}, -n_2^{\#}) $ is conformal,
which is a contradiction since ${\bf u}\in \Gr(T)$.

$(3)\ {\bf u}=(a, -b, c)$. For the third case we argue in a similar manner as in the second case.

Thus in all cases $D({\bf u})$ is indispensable in $\Lambda (T)_{\{1\}}$ and thus $\Lambda (T)_{\{1\}}$ is strongly robust.  \qed

\end{proof}

A different proof of Theorem \ref{curveA3} can be given using Lemmata \ref{Lemma}, \ref{circuit-ind} and \cite[Corollary 4.6]{KTV}, since  for $T=(n_1, n_2, n_3)$ the ideal $I_{\Lambda (T)_{\{1\}}}$ has codimension 2. We have preferred the above proof, as it follows the style of the remaining proofs in this article.
\newline

\section{Primitive elements and strongly robustness}
\label{primitiveRobustness}
In this section, we generalise the notion of a primitive element and that of a Graver basis and use it to give a necessary and sufficient criterion for a vertex $\{i\}$ to be a face of the strongly robust simplicial complex.
\begin{Definition}
    An element ${\bf u}\in S\subset \mathbb{Z}^n$ is called primitive in $S$ if there is no ${\bf v}\in S $, ${\bf v}\not ={\bf u} $ such that ${\bf v}^+\leq {\bf u}^+$ and ${\bf v}^-\leq {\bf u}^-$. The set of primitive elements in $S$ is denoted by $\Graver(S)$.
\end{Definition}
Primitive elements in $S=\Ker_{\mathbb{Z}}(A)$ constitute the Graver basis of $A$.

\begin{Definition} Let $T=(n_1, n_2,  \dots ,n_s)$, we define
    $$\Gr(T)^i=\{[{\bf u}]^i|{\bf u}\in \Gr(T)\}\subset \mathbb{Z}^{s-1}. $$
\end{Definition}
\begin{Proposition} \label{primitive} Let ${\bf u}\in \Gr(T).$ Then  $D({\bf u})$ is indispensable in $\Lambda (T)_{\{i\}}$ if and only if $[{\bf u}]^i$ is primitive in $\Gr(T)^i$.
\end{Proposition}
\begin{proof} Suppose  $[{\bf u}]^i$ is not primitive in $\Gr(T)^i$. Then there exists an element ${\bf v}
    \in \Gr(T)$ such that  $[{\bf v}^+]^i\leq [{\bf u}^+]^i$ and $[{\bf v}^-]^i\leq [{\bf u}^-]^i$. Set ${\bf w}={\bf u}-{\bf v}$. Then it follows that ${\bf u}={\bf v}+{\bf w}$ and $[{\bf u}]^i=[{\bf v}]^i+_c[{\bf w}]^i$. Looking at the i-components of ${\bf v}, {\bf w}$ and putting in front the negative one we get that one
    of the sums  $D({\bf u})=D({\bf v})+D({\bf w})$ or  $D({\bf u})=D({\bf w})+D({\bf v})$ is semiconformal. Thus $D({\bf u})$ is not indispensable in $\Lambda (T)_{\{i\}}$.

    Suppose
    that $D({\bf u})$ is not indispensable in $\Lambda (T)_{\{i\}}$.
        Then there is a proper semiconformal decomposition $D({\bf u})=D({\bf v})+_{sc}D({\bf w})$, note that in any decomposition like that the sum is conformal in every component except possible at the i-th component, see Proposition \ref{conformal}. But ${\bf u}\in \Gr(T)$ therefore $D({\bf u})\in \Gr(\Lambda (T)_{\{i\}})$,  thus it does not have a proper conformal decomposition. We conclude that the i-th component of ${\bf v}$ is negative and the i-th component of ${\bf w}$ is positive. From all semiconformal decompositions $D({\bf u})=D({\bf v})+_{sc}D({\bf w})$ choose one with $w_i$ smallest. Then, for these choices of
    ${\bf v}, {\bf w}$, we claim that ${\bf w}\in \Gr(T).$ Suppose not, then there is
    a proper conformal decomposition ${\bf w}={\bf w'}+_c{\bf w''}={\bf w''}+_c{\bf w'}$. Choose ${\bf w''}$ to be one with $w_i>w''_i$. But then it is easy to see that $D({\bf u})=D({\bf v}+{\bf w'})+_{sc}D({\bf w''})$, which is a contradiction, since
    $w_i>w''_i$. Thus ${\bf w}\in \Gr(T)$. Therefore $[{\bf w}]^i$ is in $\Gr(T)^i$
    and $[{\bf w}^+]^i\leq [{\bf u}^+]^i$ and $[{\bf w}^-]^i\leq [{\bf u}^-]^i$. Thus
     $[{\bf u}]^i$ is not primitive in $\Gr(T)^i$.
\qed

\end{proof}

\begin{Theorem} \label{main} Let $T=(n_1, n_2,  \dots ,n_s)$.
The strongly robust complex $\Delta _T$ contains $\{ i\}$ as a face if and only if
for every element ${\bf u}\in \Gr(T)$ the  $[{\bf u}]^i$ is primitive in $\Gr(T)^i$
 $($or equivalently if and only if
$\Graver(\Gr(T)^i)=\Gr(T)^i$ $)$.

\end{Theorem}

\begin{proof}
Suppose that $\{ i\}$ is a face of     $\Delta _T$ then $\Lambda (T)_{\{i\}}$ is strongly robust. Therefore  $D({\bf u})$ is indispensable in $\Lambda (T)_{\{i\}}$ for every element ${\bf u}\in \Gr(T)$. Then  by Proposition \ref{primitive} for every element ${\bf u}\in \Gr(T)$ the  $[{\bf u}]^i$ is primitive in $\Gr(T)^i$.

Suppose now that
for every element ${\bf u}\in \Gr(T)$ the  $[{\bf u}]^i$ is primitive in $\Gr(T)^i$.
Then Proposition \ref{primitive} implies that  $D({\bf u})$ is indispensable in $\Lambda (T)_{\{i\}}$ for every ${\bf u}\in \Gr(T)$. But \cite[Theorem 1.11]{PTV} says that all
elements in the Graver basis of  $\Lambda (T)_{\{i\}}$ are in the form $D({\bf u})$
with ${\bf u}\in \Gr(T)$. Thus $\Lambda (T)_{\{i\}}$ is strongly robust and so $\{ i\}$ is a face of     $\Delta _T$.
\qed

\end{proof}

\section{Generalized Lawrence matrices}
\label{genmat}

 In \cite[Section 2]{PTV}, Petrovi{\' c} et al generalized the notion of a Lawrence matrix, see \cite[Chapter 7]{St}, by introducing generalized Lawrence matrices. For every matrix $A$ there exists a generalized  Lawrence matrix with the same kernel up to permutation of columns, see \cite[Corollary 2.3]{PTV}. Let $(c_{1},\ldots,c_{m})\in\ZZ^{m}$ be any   vector  having full support and with the greatest common divisor of all its components equal to $1$.
Then there exist integers $\lambda_{1},\ldots,\lambda_{m}$ such that $1=\lambda_{1}c_{1}+\cdots+\lambda_{m}c_{m}$.
For any choice of $\lambda_{1},\ldots,\lambda_{m}$ and any integer $n$ we define the matrix $A(n,(c_{1},\ldots,c_{m}))= (\lambda_{1}n,\ldots,\lambda_{m}n)\in\ZZ^{1\times m}$ and the matrix
\[
C(c_{1},\ldots,c_{m})=\left( \begin{array}{ccccc}
 -c_{2} &  c_{1} &   &  &\  \\
 -c_{3} &  & c_{1} &  &\  \\
& & &  \ddots &  \\
 -c_{m} &  &  &   &\ c_{1}
\end{array} \right)
	 \in\ZZ^{(m-1)\times m}.
\]

\begin{Theorem}\label{inverse_construction}
	Let $T=(n_1, n_2,  \dots ,n_s)\in \ZZ^{1\times s}$.  Let $\cb_1,\ldots,\cb_s$ be any set of  vectors having full support and each with the greatest common divisor of all its components equal to $1$, with $\cb_j\in\ZZ^{m_j}$ for some $m_j\geq 1$. In the case that $\Delta _T=\{\emptyset\}$ each ${\bf c}_j=(c_{j1},\ldots,c_{jm_j})\in\ZZ^{m_j}$ has the  first component positive and at least one component negative, while in the case that $\Delta _T=\{\emptyset, \{i\}\}$ then the same is true for all ${\bf c}_j$ with $j\not =i$. Define  $p=1+\sum_{i=1}^s (m_i-1)$ and $q=\sum_{i=1}^s m_i$. Then the toric ideal $I_A$ is strongly robust, where
\[
{\footnotesize A=\
\left( \begin{array}{cccc}
 A_1 & \ A_2 & \cdots & \ A_s \\
 C_1 &\ 0\ & \cdots & 0 \\
 0 &\ C_2\ & \cdots & 0 \\
\vdots&\vdots & \ddots & \vdots \\
 0 &\ 0 & \cdots &\ C_s
\end{array} \right) } \in \ZZ^{p\times q},
\]
$A_j=A(n_j, (c_{j1},\ldots,c_{jm_j}))$ and $C_j=C(c_{j1},\ldots,c_{jm_j})$ for all $j=1,\ldots,s$.
	\end{Theorem}
	
	\begin{proof}
	    The matrix $A$ is a generalized Lawrence matrix, see \cite[Theorem 2.1]{PTV} with bouquet ideal the toric ideal of the monomial curve $I_T$,  with all bouquets  mixed except possibly the $B_i$, in the case that $\Delta _T=\{\emptyset, \{i\}\}$. Therefore the toric ideal is $T_{\omega}$-robust,
	    where $\omega$ is either the empty set or $\{ i\}$. In both cases $\omega$
	    is a face of the strongly robust simplicial complex  $\Delta _T$. Thus the toric ideal is strongly robust, see Theorem \ref{face}.\qed

  \end{proof}

	 \begin{Example} {\em
	 Theorem \ref{inverse_construction} provides a way to produce examples of strongly robust ideals with bouquet ideal the ideal of any monomial curve. Take for example the monomial curve with defining matrix $T=(4, 5, 6)$. The toric ideal $I_T$ is a complete intersection on $5$ thus according to Theorem \ref{curveA3} the strongly robust complex $\Delta _T$ is equal to $\{\emptyset, \{2\}\}$. Choose three integer vectors, one for each bouquet,
	 of any dimension with full support and the greatest common divisor of all its components equal to $1$ such that they   have a positive first component and at least one component negative except possibly for the second vector which may have all components positive. For example choose ${\bf c}_1=(2, -1, -2023)$,  ${\bf c}_2=(10, 2024, 7, 4)$ and ${\bf c}_3=(5, 3, -2029)$. For each vector
${\bf c}=(c_{1},\ldots,c_{m})\in\ZZ^{m}$ choose
integers $\lambda_{1},\ldots,\lambda_{m}$ such that $1=\lambda_{1}c_{1}+\cdots+\lambda_{m}c_{m}$. For example $1=0\cdot 2+(-1)\cdot(-1)+0\cdot(-2023)$, $1=(-1)\cdot 10+0\cdot 2024+1\cdot 7+1\cdot 4$ and $1=2\cdot 5+(-3)\cdot 3+0\cdot(-2029).$ Then Theorem \ref{inverse_construction}
says that the toric ideal $I_A$ is strongly robust, where
\[
{\footnotesize A=\
\left( \begin{array}{cccccccccc}
 0 & -4 & 0 &  -5 & 0 & 5 & 5 & 12 & -18 & 0 \\
 1 & 2  & 0 & 0 & 0 & 0 & 0 & 0 & 0 & 0 \\
 2023 & 0 & 2 & 0 & 0 & 0 & 0 & 0 & 0 & 0 \\
0 & 0 & 0 & -2024 & 10 & 0 & 0 & 0 & 0 & 0  \\
0 & 0 & 0 & -7 & 0 & 10 & 0 & 0 & 0 & 0   \\
0 & 0 & 0 & -4 & 0 & 0 & 10 & 0 & 0 & 0  \\
0 & 0 & 0 & 0 & 0 & 0 & 0 & -3 & 5 & 0  \\
0 & 0 & 0 & 0 & 0 & 0 & 0 & 2029 & 0 & 5  \\
\end{array} \right) } .
\]
	 }
	 \end{Example}

According to the following theorem, {\em all} strongly robust ideals with bouquet ideal the ideal of a monomial curve are produced like in the above example.
\begin{Theorem}\label{all_toric_is_gen_lawrence}
  Let $I_A$ be any toric ideal which is strongly robust and such that its bouquet ideal is the ideal of a monomial curve. Then there exists a
  generalized Lawrence matrix $A'$ such that $I_A=I_{A'}$, up to permutation of column indices, where
\[
{\footnotesize A'=\
\left( \begin{array}{cccc}
 A_1 & \ A_2 & \cdots & \ A_s \\
 C_1 &\ 0\ & \cdots & 0 \\
 0 &\ C_2\ & \cdots & 0 \\
\vdots&\vdots & \ddots & \vdots \\
 0 &\ 0 & \cdots &\ C_s
\end{array} \right) }  \in \ZZ^{p\times q},
\]
with matrices $A_j=A(n_j, (c_{j1},\ldots,c_{jm_j}))$ and $C_j=C(c_{j1},\ldots,c_{jm_j})$ for some  $T=(n_1, n_2,  \dots ,n_s)\in\ZZ^{1\times s}$, and $\cb_j\in\ZZ^{m_j}$ for some $m_j\geq 1$, for all $j=1,\ldots,s$,  are integer vectors having full support and each one with the greatest common divisor of all its components equal to $1$. In the case that $\Delta _T=\{\emptyset \}$ each ${\bf c}_j=(c_{j1},\ldots,c_{jm_j})\in\ZZ^{m_j}$ has the first component positive and at least one component negative.
 In the case that $\Delta _T=\{\emptyset, \{i\}\}$ each ${\bf c}_j$ with $j\neq i$ has the  first component positive and at least one component negative, while ${\bf c}_i$ may have all components positive.
\end{Theorem}

\begin{proof} By hypothesis, the ideal $I_A$ is strongly robust with its bouquet ideal $I_T$, where $T=(n_1, n_2,  \dots ,n_s)\in\ZZ^{1\times s}$. By Theorem~\ref{Delta} the strongly robust complex $\Delta _T$ is either $\{\emptyset \}$ or $\{\emptyset, \{i\}\}$
for one $i\in [s]$. Then the ideal $I_A$ is either $T_{\empty}$ or $T_{\{i\}}$-robust,
therefore all bouquets of $I_A$ are mixed with the possible exception of the $i$-th bouquet.

For any integer matrix $A$, there exists a generalized Lawrence matrix $A'$ such that $I_A=I_{A'}$, up to permutation of column indices, by  \cite[Corollary 2.3]{PTV}.
Since the bouquet ideal is the ideal $I_T$ and all bouquets of $I_A$ are mixed with the possible exception of the $i$-th bouquet,  the vectors  ${\bf c}_j$ have a positive first component and at least one component negative, with the possible exception of ${\bf c}_i$ which may have all components positive. \qed

\end{proof}

  \begin{Example} \label{E}{\em Consider the matrix $A=(\ab_1,\ldots,\ab_{11})\in\ZZ^{8\times 11}$, given by
 \[
{\footnotesize A= \left( \begin{array}{ccccccccccc}
 36 & 60 & 4 &  40 & 64 & 39 & 1 & 72 & 84 & 12 & 4 \\
 12 & 20  & 4 & 8 & 24 & -2 & 1 & 12 & 4 & 0 & 4 \\
 36 & 80 & 4 & 48 & 88 & 39 & 1 & 84 & 84 & 12 & 4 \\
 60 & 100 & 12 & 16 & 112 & 33 & 4 & 120 & 104 & 36 & 24  \\
 24 & 40 & 8 & 24 & 48 & -4 & 2 & 36 & 8 & 0 & 8  \\
 12 & 20 & 4 & -12 & 24 & -2 & 1 & 24 & 8 & 12 & 8  \\
 12 & 20 & 4 & -12 & 24 & -2 & 1 & 24 & 12 & 12 & 12  \\
 24 & 40 & 0 & 4 & 40 & 39 & 1 & 60 & 84 & 24 & 4
\end{array} \right) } .
\]
 Using \texttt{4ti2} \cite{4ti2} we can see that the toric ideal $I_A$ is strongly robust and a Gale transform of the matrix $A$ is
\[
{\footnotesize
\left( \begin{array}{cccc}
 5 & 40 & 779 &  -13642   \\
 -18 & -162  & -3198 & 56004 \\
 -15 & -120 & -2337 & 40926   \\
 0 & 6 & 123 & -2154   \\
 15 & 135 & 2665 & -46670    \\
 0 & 0 & 4 & -72   \\
 0 & 0 & 8 & -144   \\
 0 & -4 & -82 & 1436    \\
 0 & 0 & 0 & 1   \\
 0 & 14 & 287 & -5026   \\
 0 & 0 & 0 & -1
\end{array} \right) } .
\]

The toric ideal $I_A$ has five bouquets $B_1=\{a_1, a_3\}$, $B_2=\{a_2, a_5\}$, $B_3=\{a_6, a_7\}$, $B_4=\{a_4, a_8, a_{10}\}$,  $B_5=\{a_9, a_{11}\}$ all of them being mixed except the third one, which is non-mixed. The corresponding $\cb_B$ vectors  are: $\cb_{B_1}=(1,0,-3,0,0,0,0,0,0,0,0)$
 \[\cb_{B_2}=(0, 6, 0, 0, -5, 0, 0, 0, 0, 0, 0), \ \ \cb_{B_3}=(0,0,0,0,0,1,2,0,0,0,0), \]
 \[\cb_{B_4}=(0,0,0,3,0,0,0,-2,0,7,0), \ \ \cb_{B_5}=(0,0,0,0,0,0,0,0,1,0,-1).\]

 The bouquet ideal is the toric ideal of the matrix $A_B=(\ab_{B_1},\ab_{B_2},\ab_{B_3},\ab_{B_4},\ab_{B_5})$, that is
 \[
{\footnotesize A_B=\
\left( \begin{array}{ccccc}
 24 & 40 & 41 &  60 & 80  \\
 0 & 0  & 0 & 0 & 0 \\
 24 & 40 & 41 & 60 & 80  \\
 24 & 40 & 41 & 60 & 80  \\
 0 & 0 & 0 & 0 & 0   \\
 0 & 0 & 0 & 0 & 0  \\
 0 & 0 & 0 & 0 & 0  \\
 24 & 40 & 41 & 60 & 80
\end{array} \right) }.
\]
 The bouquet ideal $I_{A_B}$ is exactly the same as the toric ideal of the monomial curve for $T=(24, 40, 41, 60, 80)$ for
 which we know
that the strongly robust complex is $\Delta _T=\{\emptyset, \{3\}\}$. Note that only the third bouquet is not mixed and thus $I_A$ is a $T_{\{3\}}$-robust ideal  which explains why it is strongly robust.

 Then, the generalized Lawrence matrix with the same kernel, after permutation of column indices, is
 \[
{\footnotesize A'=\
\left( \begin{array}{ccccccccccc}
 24 & 0 & 40 &  40 & 41 & 0 & 60 & 60 & 0 & 80 & 0 \\
 3 & 1  & 0 & 0 & 0 & 0 & 0 & 0 & 0 & 0 & 0 \\
 0 & 0 & 5 & 6 & 0 & 0 & 0 & 0 & 0 & 0 & 0 \\
 0 & 0 & 0 & 0 & -2 & 1 & 0 & 0 & 0 & 0 & 0  \\
 0 & 0 & 0 & 0 & 0 & 0 & 2 & 3 & 0 & 0 & 0  \\
 0 & 0 & 0 & 0 & 0 & 0 & -7 & 0 & 3 & 0 & 0  \\
 0 & 0 & 0 & 0 & 0 & 0 & 0 & 0 & 0 & 1 & 1
\end{array} \right) } .
\]
 Note that the permutation of column indices to bring the vectors of the same bouquet together gives the following isomorphism of the two kernels $\phi:\Ker_{\ZZ}(A)\mapsto \Ker_{\ZZ}(A')$,  $$\phi(u_1, u_2, u_3, u_4, u_5, u_6, u_7, u_8, u_9, u_{10}, u_{11})=(u_1, u_3, u_2, u_5, u_6, u_7, u_4, u_8, u_{10}, u_9, u_{11}).$$
 }
  \end{Example}

\bigskip

{\bf Acknowledgments.} The first author has been supported by the Royal Society Dorothy Hodgkin Research Fellowship DHF$\backslash$R1$\backslash$201246. The third author has been partially supported by the grant PN-III-P4-ID-PCE-2020-0029, within PNCDI III, financed by Romanian Ministry of Research and Innovation, CNCS - UEFISCDI.

\end{document}